\documentclass[11pt,letterpaper]{article}
\usepackage{amsfonts}
\newcommand{\R}{\mathbb{R}}

\usepackage{graphicx}
\usepackage{enumerate}
\usepackage{amsmath}
\usepackage{amssymb}
\usepackage{mathtools}
\usepackage{amsthm}
\usepackage[abbrev]{amsrefs}
\usepackage{enumerate}
\usepackage[textsize=tiny]{todonotes}

\usepackage{color}

\usepackage{mathrsfs}

\BibSpec{article}{%
	+{}{\PrintAuthors}  		{author}
	+{,}{ \textit}     		{title}
	+{,}{ }             		{journal}
	+{}{ \textbf}       		{volume}
	+{}{ \parenthesize} 		{date}
	+{}{, no. } 		{number}
	+{,}{ }      	      		{conference}
	+{,}{ }      	      		{book}
	+{,}{ }            		{pages}
	+{,}{ }            	 	{note}
	+{,}{ }            	 	{status}
	+{,}{  \texttt } {eprint}
	+{.}{}              {transition}
}

\BibSpec{book}{%
	+{}{\PrintAuthors}  {author}
	+{,}{ \textit}      {title}
	+{,}{ }             {publisher}
	+{,}{ }             {place}
	+{,}{ }             {date}
	+{.}{}              {transition}
}

\BibSpec{incollection}{%
	+{}  {\PrintAuthors}                {author}
	+{,} { \textit}                     {title}
	+{.} { }                            {part}
	+{:} { \textit}                     {subtitle}
	+{,} { \PrintContributions}         {contribution}
	+{,} { \PrintConference}            {conference}
	+{,} { }                            {booktitle}
	+{,} { pp.~}                        {pages}
	+{,}{ }       		{series}
	+{}{ \textbf}       		{volume}
	+{,} { }                            {publisher}
	+{,} { }                 {date}
	+{,} { }                            {status}
	+{,} { \PrintDOI}                   {doi}
	+{,} { available at \eprint}        {eprint}
	+{}  { \parenthesize}               {language}
	+{}  { \PrintTranslation}           {translation}
	+{;} { \PrintReprint}               {reprint}
	+{.} { }                            {note}
	+{.} {}                             {transition}
}

\newcommand{\Sn}{\mathbb{ S}^{n-1}}

\newcommand{\Mat}{\operatorname{Mat}}
\newcommand{\Hom}{\operatorname{Hom}}
\newcommand{\RR}{\mathbb{R}}
\newcommand{\CC}{\mathbb{C}}
\newcommand{\HH}{\mathbb{H}}
\newcommand{\scf}{\mathbb{F}}
\newcommand{\Sp}{\mathrm{Sp}}
\newcommand{\GL}{\mathrm{GL}}
\newcommand{\SL}{\mathrm{SL}}
\newcommand{\U}{\mathrm{U}}
\newcommand{\Dens}{\operatorname{Dens}}
\newcommand{\Gr}{\operatorname{Gr}}
\newcommand{\vol}{\operatorname{vol}}

\newcommand{\wt}{\widetilde}
\newcommand{\largewedge}{\mbox{\Large $\wedge$}}
\newcommand{\Stab}{\operatorname{Stab}}

\newcommand{\St}{\operatorname{S}}

\newcommand{\vols}{\omega}
\newcommand{\volb}{\kappa}

\usepackage[tiny]{titlesec}

\renewcommand{\b}{\overline}
\renewcommand{\Re}{\operatorname{Re}}

\makeatletter
\def\cleardoublepage{\clearpage\if@twoside \ifodd\c@page\else
\hbox{}
\vspace*{\fill}
\begin{center}
\end{center}
\vspace{\fill}
\thispagestyle{empty}
\newpage
\if@twocolumn\hbox{}\newpage\fi\fi\fi}
\makeatother

\newtheorem{theorem}{Theorem}[section]
\newtheorem{lemma}[theorem]{Lemma}

\newtheorem{proposition}[theorem]{Proposition}
\newtheorem{corollary}[theorem]{Corollary}
\newtheorem{remark}[theorem]{Remark}

\newtheorem{conjecture}[theorem]{Conjecture}

\theoremstyle{definition}

\newtheorem{thmy}{Theorem}

\newcommand\blfootnote[1]{%
  \begingroup
  \renewcommand\thefootnote{}\footnote{#1}%
  \addtocounter{footnote}{-1}%
  \endgroup
}

\title{Complex and Quaternionic Analogues of Busemann's Random Simplex and Intersection Inequalities}
\author{Christos Saroglou and Thomas Wannerer} 
\begin{document}
\maketitle
\blfootnote{TW was supported by DFG grant WA 3510/3-1.}
\begin{abstract}{ In this paper, we extend two celebrated inequalities by Busemann---the random simplex inequality and the intersection inequality---to both complex and quaternionic vector spaces. Our proof leverages a monotonicity property under symmetrization with respect to complex or quaternionic hyperplanes. Notably, we demonstrate that the standard Steiner symmetrization, contrary to assertions in a paper by Grinberg, does not exhibit this monotonicity property.}
\end{abstract}
\section{Introduction}

The Busemann random simplex inequality provides a sharp lower bound on the expected volume of a random simplex  in $\RR^n$ formed by the origin and  $n$ vertices  sampled uniformly from convex bodies $K_1,\ldots, K_n\subseteq \RR^n$. It  is a  cornerstone of a beautiful  theory of affine isoperimetric inequalities for convex bodies.  From this inequality  several other important inequalities such as the Petty projection inequality, the Busemann--Petty centroid inequality, and the Busemann intersection inequality can be deduced. For further information, see the discussion  below and the books by Gardner \cite{Gardner:Tomography}, Schneider  \cite{Schneider:BM }, and Schneider--Weil \cite{S-W}. 

The  first main result of this paper  is an analogue of the Busemann random simplex inequality for complex and quaternionic vector spaces. By restricting the scalar field, these can be viewed as real vector spaces and hence  they possess convex sets in the usual sense. However, the complex and quaternionic structures give rise to  additional geometric objects, such as   complex or quaternionic subspaces. The underlying theme  of this paper is the interaction of  these geometric structures with classical convexity. 

  Throughout this paper,  we denote the scalar field by $\scf$,  which is allowed to be either the real numbers  $\RR$, the complex numbers  $\CC$, or the quaternions $\HH$.    We denote by $p$ the dimension of $\scf$ over $\R$ (hence $p\in\{1,2,4\}$) and by $\det(x_1,\dots,x_n)$ the determinant of the matrix with columns $x_1,\dots,x_n\in \scf^n$. Since we also consider matrices with quaternionic entries,  it is important to note that the term ``determinant"  always refers to the Dieudonn\'e determinant. The main properties of the Dieudonn\'e determinant are exposed in  Section~\ref{sec:det} below. 
A (real, complex, quaternionic)  ellipsoid in $\scf^n$ is by definition the image of the euclidean unit ball under an $\scf$-affine transformation. In what follows, we will write $\RR_+$ for the interval $[0,\infty)$.

\begin{theorem}\label{thm-BRS}
Let $\Phi\colon \RR_+\to \RR$ be a fixed strictly increasing  function.	For  convex bodies  $K_1,\ldots, K_n\subseteq \scf^n$, set
	$$\mathscr{B}(K_1,\ldots, K_n)=\int_{x_1\in K_1}\dots\int_{x_n\in K_n}\Phi(|\det(x_1,\dots,x_n)|)dx_1\cdots dx_n.$$ Then
	\begin{equation}\label{eq-thm-BRS}
		\mathscr{B}(K_1,\ldots, K_n)\geq \mathscr{B}(B_1,\ldots, B_n),    
	\end{equation}
	where $B_i$ is the euclidean ball with center at the origin and of volume equal to the volume of $K_i$. 
	
	Moreover, if the bodies $K_1,\ldots, K_n$ have non-empty interior, then equality holds in \eqref{eq-thm-BRS} if and only if the $K_i$ are homothetic  (real, complex, or quaternionic) ellipsoids centered at the origin. 
\end{theorem}

\begin{remark}\label{rem-BRS} As expected, the convexity can be dropped in Theorem~\ref{thm-BRS} without any significant changes in the proof (see \cite{Pfiefer:Max}). If one assumes $K_1,\ldots,K_n$ to be merely compact (or even just measurable) and of positive volume, then \eqref{eq-thm-BRS} is still true, with equality if and only if $K_1,\ldots, K_n$ are, up to sets of measure zero, homothetic ellipsoids centered at the origin. As a consequence, one concludes that convexity can be replaced with compactness in Theorem~\ref{thm-BI} below as well. Equality holds if and only if $K_1,\ldots, K_n$ are as in the statement of Theorem~\ref{thm-BI}, up to sets of measure zero.
\end{remark}

In the real case,  the above theorem is known as the Busemann random simplex inequality. Essentially, the determinant can be interpreted as $n!$ times the volume of the simplex formed by the origin and  $x_1,\ldots, x_n$.  The complex version of Theorem~\ref{thm-BRS} is discussed in Grinberg's work \cite{Grinberg:Isoperimetric}. However, there is an issue with the proof presented therein, as we will elaborate.

 The classical proof of the Busemann random simplex inequality is based on the property that Steiner symmetrization does not increase $\mathscr{B}$. More precisely, 
\begin{equation}\label{eq:counterex} \mathscr{B}(K_1,\ldots, K_n) \geq \mathscr{B}(\St_H K_1,\ldots, \St_H K_n)\end{equation} 
holds for every real  hyperplane $H\subseteq \RR^n$.  Grinberg's paper suggests that in the complex case the  same property holds with the same proof as in the real case. This is problematic, as it  contradicts the  characterization of the equality case as described in \cite[Theorem 11]{Grinberg:Isoperimetric}  and Theorem~\ref{thm-BRS}. 

Indeed, consider  $K_1=\cdots = K_n=E$, where $E$ is a complex  ellipsoid but not a Euclidean ball. In this scenario,  equality holds in \eqref{eq-thm-BRS}. However,  there exist real hyperplanes $H$ in $ \CC^n $ such that  $\St_H E$ is not a complex ellipsoid. Consequently,  the inequality \eqref{eq:counterex} cannot hold true  for such $H$. 

We will demonstrate, using a different argument than the one used in the real case, that \eqref{eq:counterex} holds for symmetrization in  \emph{complex} and \emph{quaternionic} hyperplanes.  In the quaternionic case, this requires first establishing a weak form of the Laplace expansion for the Dieudonn\'e determinant (Proposition~\ref{prop:row_exp}).

\bigskip

Our second main result is an analogue of the Busemann intersection inequality for complex or quaternionic vector spaces. 
In the following theorem, $\Gr_{m}(n,\scf)$ denotes the Grassmannian of $m$-dimensional $\scf$-linear subspaces of $\scf^n$ and $dE$ denotes integration with respect to the unique Haar probability measure on the Grassmannian.
Let $\volb_n$ denote the volume of the $n$-dimensional euclidean unit ball in $\RR^n$. 
The Lebesgue measure of a set $A\subseteq \RR^n$ is denoted by $|A|$. 

\begin{theorem} \label{thm-BI} Let $m\in \{1,\ldots, n-1\}$ and let 
		 $K_1,\ldots, K_m\subseteq \scf^n$ be convex bodies with non-empty interior.  If $p\in \{1,2,4\}$ denotes the dimension of $\scf$ over $\RR$, then
		\begin{equation}\label{eq:thm-BI} |K_1|\cdots |K_m| \geq \frac{(\kappa_{np})^m}{(\kappa_{mp})^n} \int_{\Gr_m(n,\scf)}  |K_1\cap E|^{n/m} \cdots |K_m\cap E|^{n/m}  dE,\end{equation}
		where $|K_i\cap E|$ denotes the Lebesgue measure of $K_i\cap E$  in $E$. 
		
		For $m=1$ equality holds if and only if the $K_i$ are invariant under multiplication by scalars of unit norm.
		If $m\geq 2$, then   equality holds in \eqref{eq:thm-BI} if and only if the $K_i$ are homothetic  (real, complex, or quaternionic) ellipsoids centered at the origin. 

\end{theorem}

We remark that, due to a straightforward application of H\"older's inequality, Theorem \ref{thm-BI} in the case $K_1=\dots=K_m$ is in fact equivalent to the general case of Theorem \ref{thm-BI}.

The Busemann intersection inequality was initially proven by Busemann in  \cite{Busemann:Concurrent} for $m=n-1$ and later in \cite[Equation (9.4)]{BusemannStrauss:Area} for general $m$. Grinberg rediscovered the  general case  in \cite{Grinberg:Isoperimetric}. The paper also introduces  the complex version of the inequality. In all cases, the intersection inequality can be derived from the Busemann random simplex inequality via a useful identity  known as the linear Blaschke--Petkantchin formula.

As mentioned, Theorem~\ref{thm-BI} for $m\geq2$ can be formulated as 
\begin{equation}\label{eq:BI_same}|K|^m \geq \frac{(\kappa_{np})^m}{(\kappa_{mp})^n} \int_{\Gr_m(n,\scf)}  |K\cap E|^{n}  dE,\end{equation}
with equality if and only if $K$ is a centered ellipsoid. In the real case, the quantity on the right-hand side is called the $m$th \emph{dual affine quermassintegral} and is denoted by $\wt \Phi_m(K)$. It is well known that $\wt\Phi_m(K)$ is invariant under volume-preserving linear transformations. This was proved by Grinberg~\cite{Grinberg:Isoperimetric}, who also observed affine invariance over the complex numbers. We will reprove these results and establish the analogous property over the  quaternions in Section~\ref{sec:invariance}, along with invariance of the $m$th \emph{affine quermassintegral}
$$
\int_{\Gr_m(n,\scf)} |P_E K|^{-n} dE.$$
Here $P_E\colon \scf^n\to E$ denotes the orthogonal projection and $|P_E K|$ denotes the (euclidean) volume of $P_EK$ in $E$.  In fact, our proof will show how to construct these integrals using  only  a choice of Lebesgue measure on a vector space. 

The results of this work suggest to formulate a conjecture by Lutwak~\cite{Lutwak:affineQuermass}, which was recently confirmed by Milman--Yehudayoff \cite{MilmanYehudayoff:Affine}, also over the  complex numbers and  the quaternions: 

\begin{conjecture}\label{conj} For every convex body $K$ in $\scf^n$ with non-empty interior
	\begin{equation}\label{eq:conj} |K|^{-m} \geq  \frac{(\kappa_{mp})^n}{(\kappa_{np})^m} \int_{\Gr_m(n,\scf)} |P_E K|^{-n} dE \end{equation}
	with equality if and only if $K$ is a (real, complex, or quaternionic) ellipsoid.	
\end{conjecture}
It is not difficult to see that  Conjecture~\ref{conj} is true in the specific case where $m=1$ and $K$ is the unit ball of a (complex or quaternionic) norm (Proposition~\ref{prop:special}). 
In the real case, Milman and Yehudayoff \cite{MilmanYehudayoff:Affine} demonstrated that the integral in \eqref{eq:conj} is monotone under  Steiner symmetrization.  This is not true in the complex and quaternionic cases, since ellipsoids are only extremizers if they are complex or quaternionic.
Thus a proof of Conjecture~\ref{conj} based on symmetrization will require a different method. 

As in the real case, the conjecture directly  implies the isoperimetric  inequalities 
\begin{equation}\label{eq:iso}\frac{\kappa_{np}}{\kappa_{mp}} \int_{\Gr_m(n,\scf)} |P_E K| dE \geq   |K|^{n/m}.\end{equation}
While these inequalities are highly compelling, they remain  open over the complex numbers and the quaternions.

\subsection{Relation to other work}

The intersection body of a convex body $K\subseteq \RR^n$  was introduced by Lutwak \cite{lutwak:intersection} as the star-shaped body $IK$ determined by 
$$ |IK\cap H^\perp| = |K\cap H|$$
for every linear hyperplane $H\subseteq \RR^n$. 
The Busemann intersection inequality \eqref{eq:BI_same} for $m=n-1$  yields a sharp upper bound for the volume of the intersection body. Replacing real by complex hyperplanes and requiring $IK$ to be invariant under multiplication by unit complex numbers, Koldobsky--Paouris--Zymonopoulou \cite{KPZ:IB} introduced complex intersection bodies.  This definition was latter extended by Dann--Zymonopoulou \cite{DZ:IB} to quaternionic vector spaces. As in the real case, the Busemann intersection inequality of Theorem~\ref{thm-BI} yields a sharp upper bound in the volume of the complex and quaternionic intersection bodies.  The concept of intersection body exists also within  the  $L_p$-Brunn--Minkowski theory, see, e.g., \cite{Berck,LZ:BS, HaberlLudwig:Intersection, GardnerGiannopoulos}. An $L_p$ version of the complex intersection body was recently introduced by Ellmeyer and Hofst\"{a}tter \cite{Ell-Hof}. 

The classical Busemann random simplex inequality is well known to translate for $\Phi(x)=x$  into a sharp inequality for the the volume of the centroid body of a convex body. 
Inspired by the work of Abardia and Bernig \cite{Aba,Aba-Be}, an analogue of the centroid body for complex vector spaces was proposed by Haberl \cite{Haberl}.  The volume of this body, however, does not seem to be related to the functional $\mathscr{B}$.  

The integrals in \eqref{eq:iso} arise in complex and quaternionic integral geometry as  important examples of unitarily invariant valuations, see \cite{Alesker:hl,BF:hig,BS:plane}.

\section{Preliminaries}\label{sec-prelim}

In this paper euclidean inner products are denoted by $\langle u,v\rangle$  and for the corresponding norms we write $\|v\|$.  
Following the standard reference \cite{S-W},  let $$\volb_n= \frac{\pi^{n/2}}{\Gamma(\frac n2 + 1)} $$
 denote the volume of the $n$-dimensional euclidean unit ball $B^n\subseteq \RR^n$ and we write $$\vols_{n}= n\volb_n= \frac{2\pi^\frac{n}{2}}{\Gamma(\frac n2)}$$
  for the surface area of  the $(n-1)$-dimensional sphere $S^{n-1}=\partial B^n$.

\subsection{Convex geometry}

Let $E$ be a $k$-dimensional linear subspace in $\RR^n$ and let $A\subseteq \RR^n$ be compact. The symmetrization of $A$ with respect to the subspace $E$ is defined as follows. For every $(n-k)$-dimensional affine subspace $F\perp E$ that meets $A$, consider the euclidean ball  $B_F\subseteq F$ (possibly a singleton) with center in $E$ and  volume equal to the volume of $A\cap F$. The union of all such balls is the symmetrization of $A$ with respect to $E$ and denoted by $\St_E A$, see, e.g., \cite[Section 9.2]{BZ}. For $k=n-1$ one obtains the Steiner symmetrization and for $k=1$ the Schwarz symmetrization.

Observe that $\St_E A$ can be obtained as a limit of Steiner symmetrizations. In particular, $\St_E A$ is convex if $A$ is convex.

A generalization of a remarkable theorem due to Klain \cite{Klain:Steiner} will be used subsequently.

\begin{theorem}[{Bianchi--Gardner--Gronchi \cite[Theorem 6.9]{BGG:Convergence}}]\label{Bianchi-Gardner-Gronchi}
	Let $i\in\{1,\dots, n-2\}$. Let, also, $E_1,\dots, E_m\subseteq \R^n$ be $i$-dimensional linear subspaces  such that $E_1^\perp+\dots +E_m^\perp=\R^n$,
 and that the set $\{E_1,\dots,E_m\}$ cannot be partitioned into two mutually orthogonal non-empty subsets. Let $\{F_i\}_{i=1}^\infty$ be a sequence from $\{E_1 ,\dots,E_m \}$, such that each $E_i$ appears infinite many times in the sequence. Then, for a convex body $C$, the sequence
	$\St_{F_i}\circ\dots\circ \St_{F_1}(C)$ converges in the Hausdorff metric to a euclidean ball.
\end{theorem}

To characterize when equality occurs in our inequalities, we will use the following  elementary characterization of euclidean balls.  We refer to \cite{DLL:Ellipsoids} or \cite[Lemma 2]{Groemer:OnSomeMeanValues}) for a related characterization of ellipsoids.

\begin{lemma}
\label{thm:ellipsoids} Let $K$ be a convex body in $\RR^n$, such that for each direction $u\in\Sn$, there is a hyperplane $H$ which is perpendicular to $u$ and $K$ is symmetric with respect to $H$. Then $K$ is a euclidean ball.
\end{lemma}

\subsection{Linear algebra over non-commutative fields}\label{sec:linAlg}

The basic concepts of linear algebra extend with almost no change to the situation where the  field is not necessarily commutative.  In this paper, by a vector space $V$ over a division ring $k$ we will always understand a right vector space over $k$. This is  a  group $(V,+)$ together with a scalar multiplication 
$$ \cdot \colon V\times k\to V $$ 
such that  the properties 
\begin{gather*}
	(v\cdot a)\cdot b= v\cdot (ab), \quad (v+w)\cdot a= v\cdot a+ w\cdot a,\\
	v\cdot (a+b)= v\cdot a+ v\cdot b,\quad v\cdot 1 = v
\end{gather*} 
are satisfied. The definition and properties of subspaces, span, linear independence, basis, and dimension carry over without change to this more general setting, see, e.g., \cite[Chapter 1]{Artin}. A linear map $f\colon V\to W$ is defined as in the commutative setting. If $v_1,\ldots, v_n$ is a basis of $V$ and $w_1,\ldots, w_m$ is a basis of $W$, then the matrix $A=(a_{ij})\in \Mat(m\times n,k)$
is defined by 
$$ f(v_j)= \sum_{r=1}^m  w_j\cdot a_{rj},\quad j=1,\ldots,n.$$ Here, $\Mat(m\times n,k)$ stands for the set of $m\times n$ matrices with entries from $k$. We also abbreviate $\Mat(n,k)=\Mat(n\times n,k)$. 
Observe that if $v\in V$ has the coordinates $X=(x_1,\ldots, x_n)^t$, then $f(v)$ has the coordinates $AX$. 
If $g\colon U\to V$ is another linear map with matrix $B=(b_{kl})$, then 
the matrix of $f\circ g$ is the product of the matrices 
$AB=(\sum_{r=1}^n a_{ir}b_{rl})$, as usual.

An important difference with the commutative case that arises already at this elementary level is that the set $\Hom(V,W)$ of  linear maps  fails in general to be a vector space over $k$. Also the definition of the determinant requires a modification as we will discuss in the next subsection.

In this article, we are ultimately interested in  only in the  division rings of the reals, the complex numbers, and the quaternions. We use the letter $\scf$ to denote either $\R$,  $\CC$, or $\HH$. Recall that every $x\in \HH$ can be uniquely expressed as 
$$ x= x_0 + x_1 i + x_2 j + x_3 k,$$
where  $x_0,x_1,x_2,x_3\in\RR$ and 
$$ i^2= j^2=k^2 = ijk=-1.$$
Conjugation in $\HH$ is defined by $\b x= x_0 - x_1 i -x_2 j - x_3 k$ and satisfies 
$$\b{xy} = \b y \; \b x,\quad x,y\in\HH.$$
The real part of $x$ is denoted by $\Re x=x_0$. Observe that $x+ \b x = 2\Re x$. 
The  absolute value or norm of a quaternion is defined by $|x| = \sqrt{x\b x}$ and a euclidean inner product on $\HH$ is defined by polarization, 
$$ |x+y|^2 =|x|^2 + 2\langle x,y\rangle + |y|^2.$$
Note that 
\begin{equation}\label{eq:Re} \langle x,y\rangle = \Re(\b x y ) =\Re(x\b y)= \Re(\b yx)\end{equation}
and consequently,
$$ \langle xw, yw\rangle = \langle wx,wy\rangle = |w| \langle x,y\rangle.$$ 

A (hyper-)hermitian inner product $V\times V\to \scf$ on a vector space $V$ over $\CC$ (or $\HH$) satisfies by definition the properties
\begin{gather*}
	( u+v,w) =( u,w)+ ( v,w), \quad ( v\cdot a, w) = \overline a ( v,w), \quad 
	( w,v) = \overline{( v,w)},
\end{gather*}
and $( v,v) \geq 0$ with equality if and only if $v=0$. Such an inner product defines a norm  $\|v\| = \sqrt{( v,v)}$ and a  euclidean inner product $\langle u,v\rangle = \Re(u,v)$. It follows from \eqref{eq:Re} that $\|v\cdot a\| = |a|\|v\|$ for all $ a\in \scf$ and $v\in V$. 
An inner product space is a vector space $V$ equipped with a euclidean, hermitian, or hyperhermitian inner product.
The standard hyperhermitian inner product on $\HH^n$ is 
$$ ( v,w) = \sum_{i=1}^n \overline{v_i}w_i.$$
The adjoint of $f\in \Hom(V,W)$ exists and satisfies 
$$ \langle v, f^* (w)\rangle = \langle f(v), w\rangle, \quad v\in V, \ w\in W.$$
If the matrices of $f$ and $f^*$ are considered  with respect to orthonormal bases of $V$ and $W$, then $A^*= \overline A^t$, as usual.

The division rings $\RR$, $\CC$, and $\HH$ give rise to the classical groups.
We denote by $\GL(V)$ the general linear group and set  $\GL(n,\scf)= \GL(\scf^n)$. The special linear group will be defined in the next subsection after we have discussed the determinant. If $V$ is an inner product space, we define
$$\U(V)= \{ f\in \GL(V)\colon f\circ f^*= I\}$$
and call it the unitary group. This provides a convenient uniform terminology for the classical groups  $\mathrm O(n)= \U(\RR^n)$,  $\U(n)=\U(\CC^n)$, and $\Sp(n)= \U(\HH^n)$. Note that the unitary group acts transitively on the set of orthonormal basis of $V$. 

\subsection{Non-commutative determinants} \label{sec:det}

Dieudonn\'e \cite{Dieudonne:Det} has extended the theory of determinants to fields  that need not be  commutative. Let $k^\times$ denote the multiplicative group of non-zero elements of a division ring $k$ and let  $[k^\times, k^\times]$ be the commutator subgroup. To the group $k^\times /[k^\times, k^\times]$ we adjoin a zero element $0$ with the obvious multiplication rules. 
Note that if $k$ is commutative, then $[k^\times, k^\times]=\{1\}$ and the Dieudonn\'e determinant coincides with the usual determinant.  We therefore do not notionally distinguish  between the classical and the Dieudonn\'e determinant in the following. 

Let us write $[a]$ for  the coset  of $a\in k^\times $  in $k^\times /[k^\times,k^\times]$ and put moreover $[0]=0$. 
\begin{theorem}[{\cite[Section IV.1]{Artin}}]\label{thm:Ddet}
	The Dieudonn\'e determinant $$\det \colon \Mat(n,k)\to  k^\times /[k^\times, k^\times] \cup \{0\} $$ is uniquely characterized by the following properties:
	\begin{enumerate}[(a)]
		\item If  $A'$ is obtained from $A$ by multiplying one row from the left by $\mu$, then $\det A' = [\mu] \det A$
		\item If  $A'$ is obtained by adding a left multiple of a row of $A$ to another  row, then $\det A' =  \det A$.
		\item The unit matrix has determinant $1$. 
	\end{enumerate}
Moreover, the Dieudonn\'e determinant has the following properties:
\begin{enumerate}[(1)]
		\item $\det (A)=0$ if and only if $A$ is singular.
		\item \label{eq:detAB} $\det(AB)= \det(A) \det(B)$.
		\item {If $A'$ is obtained from $A$ by interchanging two rows or two columns, then $\det A' = [-1] \det A$. }
		\item  If  $A'$ is obtained by adding a right multiple of a column of $A$ to another  column, then $\det A' =  \det A$.
		\item  If  $A'$ is obtained from $A$ by multiplying one column from the right by $\mu$, then $\det A' = [\mu] \det A$
	\end{enumerate}
\end{theorem}

The Dieudonn\'e determinant is in general not linear in each row, but satisfies only a version of subadditivity, see \cite[Theorem 4.3]{Artin}.  A recent paper \cite[Remark 1]{Pop:Quantum} asserts that there exists no row expansion for the Dieudonn\'e determinant. Formally this is correct, since there is in general no addition defined on $k^\times/[k^\times, k^\times]$.  An earlier paper by Brenner \cite[Theorem 4.3]{Brenner} claims the existence of  a row expansion, but he is not explicit about the meaning of his notation. 

 The following proposition, which will be crucial in our proof of the quaternionic  Busemann  random simplex inequality, can be interpreted as a weak form of row expansion. We suspect that it is the correct way to interpret the results of Brenner (Theorems 4.3 and 4.5 in \cite{Brenner}). 
\begin{proposition}\label{prop:row_exp}Let $A=(a_{ij}) \in \Mat(n,k)$. Then there exist elements $\lambda_1,\ldots, \lambda_n\in k$ that do not depend on the first row of $A$  such that 
	\begin{equation}\label{eq:row_exp}\det A = [ a_{11} \lambda_1 + \cdots + a_{1n}\lambda_n] \end{equation} 
and  $$[\lambda_i] = \det A(1,i),$$ where $A(i,j)\in \Mat(n-1,k)$ is the matrix $A$ with the $i$th row and $j$th column deleted. 
\end{proposition}
\begin{proof} The basic idea is to perform elementary row operations on all rows except the first one. 	By adding suitable multiples of a row to another, we obtain from  $A$ a matrix $A'$,  such that one of the following holds: (1) The matrix $A'$ has a zero row; or  (2) There exists a permutation matrix $P=P_\sigma$ such that 
	$$ A'P= \begin{pmatrix} a_{1,\sigma(1)} & \cdots  & a_{1, \sigma(n-1)} & a_{1,\sigma(n)} \\
		  y & & D &  
		  \end{pmatrix}, $$
	  where  $\sigma$ is a permutation, $D=\operatorname{diag}(d_1,\ldots, d_{n-1})$ is a diagonal matrix with non-zero diagonal entries, and $y$ is a column vector of length $n-1$. 

In the first case,  independently of the entries in the first row,  the matrix $A$ is always singular. Thus we set $\lambda_1=\cdots = \lambda_n=0$. 

In the second case, by adding suitable multiples of the columns $2,\ldots, n$  of $A'P$ to the first column, we obtain the matrix 
$$ A'' =  \begin{pmatrix} b & a_{1\sigma(2)} & \cdots  & a_{1, \sigma(n)}  \\
	0 & & D &   
\end{pmatrix}, $$
where 
$$ b=a_{1,\sigma(1)} -  a_{1\sigma(2)}d_1^{-1} y_1- \cdots  - a_{1,\sigma(n)}d_{n-1}^{-1} y_{n-1}.$$
By Theorem~\ref{thm:Ddet}\eqref{eq:detAB}, one has $$\det A = \det A' = (\det P)^{-1} \det A'' = (\det P)^{-1}[b]\det D.$$ Therefore setting 
\begin{equation}\label{eq:def_lambda} \lambda_{\sigma(1)}=\mu \quad \text{and} \quad  \lambda_{\sigma(i)} = - d_{i-1}^{-1} y_{i-1}\mu  \quad \text{for }\ i=2,\ldots, n
\end{equation}
 with $[\mu]= (\det P)^{-1}\det D $
 yields \eqref{eq:row_exp}. 
\end{proof}

We will need a more precise description of how the scalars $\lambda_1,\ldots,\lambda_n$ in Proposition~\ref{prop:row_exp} can be chosen.

\begin{lemma}\label{lemma:row_exp}
	In the setting of Proposition~\ref{prop:row_exp}, suppose that the submatrix $A(1,1)$ is non-singular. Let us write $z=(a_{i1})_{i=2}^n\in k^{n-1}$ for the first column of $A$ without the first entry. Let all other entries of $A$ be fixed.  Then $\lambda_1$ can be chosen independently of $z$ and the scalars $\lambda_2,\ldots, \lambda_n$ can  be chosen so that  there exists a non-zero element $\mu\in k$ such that  $\lambda_2(z)\cdot \mu^{-1}, \ldots,\lambda_n(z)\cdot \mu^{-1} $ are linearly  independent linear functions of $z$. 
\end{lemma}
\begin{proof}
	This follows immediately from the proof of Proposition~\ref{prop:row_exp}. Indeed, the fact that $A(1,1)$ is non-singular implies that the permutation $\sigma$ can be chosen so that $\sigma(1)=1$. Since $y$ is the result of row operations applied to $z$, the coordinates of $y$ are linearly independent linear functions of $z$. The claimed properties of the $\lambda_i$ are now obvious from   definition \eqref{eq:def_lambda}. 
\end{proof}

If $f$ is an endomorphism of $V$ with matrix $A$ with respect to some basis, we define 
$ \det f = \det A$. By Theorem~\ref{thm:Ddet}\eqref{eq:detAB} this definition does not depend on the choice of basis.

Let us consider the special case $k =\HH$. One can show that the commutator subgroup $[\HH^\times, \HH^\times]$ consists precisely of quaternions of unit norm, see, e.g., \cite[Lemma 8]{Aslaksen:Determinant} for a proof. In particular, $-1$ belongs to the commutator subgroup and thus permutation matrices have Dieudonn\'e determinant $1$. A direct consequence is that over the quaternions the Dieudonn\'e determinant has the additional property that interchanging two rows or two columns does not change the determinant. In  many other respects as well, the Dieudonn\'e determinant behaves similarly to  the absolute value of a real or complex determinant. For an introduction to the Dieudonn\'e determinant that emphasizes this perspective, see \cite{Alesker:non-comm}.

Note  that  even for quaternionic $(2\times 2)$-matrices the identity  $\det A^t = \det A$ is not true, see \cite[Section 4.2]{LinYu:Ddet} for a counterexample.  However, the following is true.
\begin{proposition} \label{prop:adjoint}
	$\det A^*= \det A$ for $A\in \Mat(n,\HH)$. 
\end{proposition}
\begin{proof}Since  (left) elementary row operations on $A$ correspond to (right) elementary column operations on $A^*$, it suffices to consider diagonal matrices. Since $a=\b a $ modulo  $[\HH^\times, \HH^\times]$, the claim follows.
\end{proof}
\begin{corollary}\label{cor:det_unitary} If $V$ is an inner product space, then
	$|\det f|= 1$ for $f\in \U(V)$. 
\end{corollary}

We define the special linear group over $\scf$ as 
$$ \SL(n,\scf)= \{ A\in \GL(n,\scf)\colon \det A=1\}$$
Also in the quaternionic case, the transformations belonging to this group are volume-preserving: 

\begin{proposition}\label{prop:real_det} In any   quaternionic vector space $V$
	\begin{equation}\label{eq:Ddet_real}|\det f|^4 = \det f_\RR,\end{equation}
	where $f_\RR$ the $\RR$-linear map obtained through the restriction of the scalar field.
\end{proposition}
\begin{proof}
	Let $(v_1,\ldots, v_n)$ be a basis of $V$ as vector space over $\HH$. Then 
	$$ (v_1,\ldots, v_n, v_1\cdot i,\ldots, v_n\cdot i, v_1\cdot j,\ldots, v_n\cdot j, 
	v_1\cdot k,\ldots, v_n\cdot k)$$
	is a basis of $V$ as a vector space over $\RR$. Equation \eqref{eq:Ddet_real} thus follows at once from \cite[Theorem 9]{Aslaksen:Determinant}. 
\end{proof}

\begin{lemma}\label{lemma-determinant}
	Let $A\in \Mat(n,\scf)$. Define the linear map
	\begin{align*} R_m\colon & \Mat(m\times n,\scf)\to \Mat(m\times n,\scf)\\
		&  M\mapsto MA,
		\end{align*}
	where $\Mat(m\times n, \scf)$ is regarded as a real vector space. Then, $|\det R_m|=|\det A|^{mp}$. 
\end{lemma}
\begin{proof} First, note that  $|\det R_m|= |\det R_1|^m$. Consider the  $\scf$-linear map $f\colon \Mat(n\times 1,\scf)\to \Mat(n\times 1,\scf)$, $f(X)=A^* X$. Clearly, one has $ R_1= *\circ f\circ *$. Therefore $\det R_1   = \det f_\RR$ and  the claim follows from Proposition~\ref{prop:real_det} and Proposition~\ref{prop:adjoint}.
\end{proof}

\subsection{Invariant smooth  measures}

Recall that a (strictly positive) smooth measure on a smooth manifold $X$ is a signed Borel measure $\mu$ such that in each coordinate chart $(U, \phi)$ there exists a (strictly positive) smooth function $f_U\colon \phi(U)\to \RR$ such that 
$$ \int_X h\, d\mu=  \int_{\phi(U)}  h(\phi^{-1}(x)) f_U(x)\, dx$$ 
for all continuous function $h\in C_c(U)$ with compact support contained in $U$. Note that if $(V,\psi)$ is another coordinate coordinate chart, then 
$$ f_U(x) =  g_{UV}(x)\,  f_V((\psi \circ \phi^{-1})(x)) , \quad x\in \phi(U\cap V),$$
with $g_{UV}(x)=|\det d(\psi\circ\phi^{-1})|(x)$,
by the change of variables formula.

The functions $g_{UV}$ coincide with the transition functions for the  density line bundle $\Dens(TX)$, see, e.g., \cite{Nicolaescu:Lectures} or \cite{Lee:Smooth}.
Thus smooth measures are in bijective correspondence with the smooth densities on $X$, i.e., smooth sections of   $\Dens(TX)$. We remind the reader that the fiber of $\Dens(TX)$   over a point $x\in X$ consists of densities on the tangent space $T_x X$ at this point.   Here a density  on an $n$-dimensional real vector spaces $V$ is a function on  the $n$th exterior power of $V$,  $\omega\colon \largewedge^n V\to \RR$, satisfying $\omega(t v_1\wedge \cdots \wedge v_n)= |t|\omega(v_1\wedge \cdots \wedge v_n)$ for $t\in \RR$. We denote by $\Dens(V)$ the $1$-dimensional real vector space of densities on $V$.  Below we will use the natural isomorphisms
\begin{gather*}
	\Dens(V^*) \simeq \Dens(V)^*,\\\Dens(V/E)\simeq \Dens(V)\otimes \Dens(E)^*.
\end{gather*} 

Densities can be pulled back under diffeomorphisms  $f\colon X\to X$   so that 
\begin{equation}\label{eq:pullback_density} \int_X f^* h  \cdot  f^* \rho  =\int_X h\cdot \rho\end{equation}
holds for every compactly supported continuous function $h$. If $G$ is a Lie group  that acts smoothly on $X$, then let us write $\lambda_g(x)= g\cdot x$. By \eqref{eq:pullback_density} a  $G$-invariant density $\rho$, i.e., a density satisfying $\lambda_g^* \rho = \rho$ for all $g\in G$, corresponds to a $G$-invariant  smooth measure on $X$ and vice versa. 

We construct now  an $\U(n,\scf)$-invariant strictly positive smooth measure on the Grassmannian $\Gr_m(n,\scf )$. There are various ways to do this and they are  well known, at least in the case $\scf=\RR$. The path we are  going take has the advantage that it allows us to easily prove the invariance of the (dual) affine  quermassintegrals in Section~\ref{sec:invariance}.

	Let $V$ be an $n$-dimensional vector space over $\scf$. The Grassmannian  $\Gr_m(V)$ is the set $m$-dimensional linear subspaces in $V$. We write $\Gr_m(n,\scf)$ for $\Gr_m(\scf^n)$.  
It is well-known that $\Gr_m(V)$ is a smooth manifold and that the general linear group $\GL(V)$ acts smoothly and transitively on $\Gr_{m}(V)$. 
We  will need a good description of the fibers of the density bundle.
If $V$ is a vector space over $\scf$, we write $\Dens(V)=\Dens(V_\RR)$ for the densities on $V$ viewed as a real vector space. 
If $G$ acts on $X$ and $x\in X$ is  a point, then $\Stab_G(x)=\{g\in G\colon g\cdot x=x\}$ denotes the stabilizer of $x$.  

\begin{proposition}\label{prop:densities}
	For each  $E\in \Gr_{m}(V)$ 
	there exists a  $\Stab_{\GL(V)}(E)$-equivariant isomorphism of real vector spaces
	$$\Dens(T_E \Gr_m(V)) \simeq  \Dens(V)^{\otimes m} \otimes \Dens(E^*)^{\otimes n }.$$
\end{proposition}

The proof of Proposition~\ref{prop:densities} will be given below  based on two lemmas. Let us first record the following consequence.

\begin{corollary}
	Up to a factor of proportionality, there  exists a unique $\U(n,\scf)$-invariant smooth measure on $\Gr_m(n,\scf)$. 
\end{corollary}

\begin{proof} It is a well-known,  simple fact that the $G$-invariant sections of a $G$-equivariant vector bundle are in  bijective correspondence with the $\Stab_G(p)$-invariant elements of the fiber over a point $p$. 
	
	Fix a subspace $E\in \Gr_m(n,\scf)$. 
	By the above, the space of $\U(n,\scf)$-invariant smooth sections of the density bundle of $\Gr_m(n,\scf)$ is isomorphic to the subspace of $\Stab_{\U(n,\scf)}(E)$-invariant elements of $\Dens(T_E \Gr_m(n,\scf))$. 	By Proposition~\ref{prop:densities} the action of $\Stab_{\U(n,\scf)}(E)$ on the fiber is trivial.
\end{proof}

\begin{lemma} For each  $E\in \Gr_{k}(V)$ 
there exists an  $\Stab_{\GL(V)}(E)$-equivariant isomorphism of real vector spaces
$$ T_E \Gr_k(V) \simeq  \Hom(E,V/E).$$ 
\end{lemma}
\begin{proof} 
	It is well-known that for every $G$-homogeneous space $X$ and every $p\in X$ there exists a $\Stab_G(p)$-equivariant isomorphism 
	$T_p X\simeq \mathfrak g/ \mathfrak h$, where $\mathfrak g$ is the Lie algebra of $G$ and $\mathfrak h$ is the Lie algebra of stabilizer. 
	
	In the situation of the lemma, $\mathfrak g= \Hom(V,V)$ {is the real vector space of $\scf$-linear endomorphisms of $V$} and 
	$$\mathfrak h=\{ f\in \Hom(V,V)\colon f(E)\subseteq E\}.$$ Hence 
$\mathfrak g/\mathfrak h\simeq \Hom(E,V/E)$ as desired. \end{proof}

\begin{lemma}
Let $V,W$ be vector spaces over $\scf$ and let $n=\dim V$ and $m=\dim W$.  Then there exists a $\GL(V)\times \GL(W)$-equivariant isomorphism 
$$ \Dens(\Hom(V,W)) \simeq  \Dens(V^*)^{\otimes m } \otimes \Dens(W)^{\otimes n}.$$
\end{lemma}
\begin{proof}
	Let $\omega_V\in \Dens(V)$ and $\omega_{W^*}\in \Dens(W^*)$ be positive densities. 
 The other two cases being parallel this one, let us consider only the case $\scf=\HH$.   	Let $v_1,\ldots, v_n$ be a basis of $V$ and let $\alpha_1,\ldots, \alpha_m$ be a basis of $W^*$ such that 
 $$ \omega_V(v_1,v_1\cdot i, v_1\cdot j, v_1\cdot k, \ldots, v_n,v_n\cdot i, v_n\cdot j, v_n\cdot k)=1$$ 
 and 
 $$ \omega_{W^*}(\alpha_1,\alpha_1\cdot i, \alpha_1\cdot j, \alpha_1\cdot k, \ldots, \alpha_m,\alpha_m\cdot i, \alpha_m\cdot j, \alpha_m\cdot k)=1.$$ 
 Define 
 $$ M(\omega_V, \omega_{W^*})=\{ f\in \Hom(V,W)\colon \max_{i,j} |\alpha_i(f(v_j))|\leq 1\}.$$
  and $\theta(\omega_V, \omega_{W^*})\in \Dens(\Hom(V,W))$ by stipulating
 $ \theta(M)=1$. 
Note that by Lemma~\ref{lemma-determinant} and Proposition~\ref{prop:real_det} the definition of $\theta(\omega_V, \omega_{W^*})$ is independent of the choice of bases of $V$ and $W^*$. 
 Let $${\omega_{V^*}}^{\otimes {m}} \otimes  {\omega_{W}}^{\otimes {n}}\mapsto  \theta(\omega_V, \omega_{W^*}),$$ where $\langle \omega_{V^*}, \omega_V\rangle=1$ and 
 $\langle \omega_{W^*}, \omega_W\rangle=1$, define an isomorphism between the spaces of densities.
Let $(g,h)\in \GL(V)\times \GL(W)$ and $\eta\in \Dens(\Hom(V,W))$. On  the one hand,  by Lemma~\ref{lemma-determinant},
$$ (g,h)\cdot \eta =  |\det g|^{4m}|\det h|^{-4n} \eta.$$
On the other hand, by Proposition~\ref{prop:real_det}, 
\begin{align*} (g,h) \cdot ({\omega_{V^*}}^{\otimes {m}}\otimes  {\omega_{W}}^{\otimes {n}})  &=(g\cdot \omega_{V^*})^{\otimes {m}} \otimes ( h\cdot \omega_{W})^{\otimes {n}}\\
&=  |\det g|^{4m} |\det h|^{-4n} ({\omega_{V^*}}^{\otimes {m}}\otimes  {\omega_{W}}^{\otimes {n}}) .
\end{align*}
Therefore the isomorphism is  $\GL(V)\times \GL(W)$-equivariant.

\end{proof}

A theory of invariant measures can be developed in much greater generality. It will be convenient for us to employ a uniqueness result attributed to H.~Weyl. Let us recall from the treatise by Nachbin \cite{Nachbin} the basic terminology of the subject. 

	Let $X$  be homogeneous space under a locally compact group  $G$. For $g\in G$ and every continuous function  $f\in C(X)$ define 
$ g\cdot f(x)= f(g^{-1} x)$. A positive linear functional $\mu :C_c(X)\to\R$ is called a 
relatively invariant positive  integral with modulus $\Delta:G\to\R$, if for every $g\in G$ there exists a number $\Delta(g)>0$ such that
$$\mu(g\cdot f)=\Delta(g)\mu(f),$$
for any $g\in G$ and for any $f\in C_c(X)$.

The following uniqueness result holds.
\begin{theorem}[{\cite[Theorem 1 of Section 4, Chapter III]{Nachbin}}]\label{thm-uniqueness}
	Let $X$  be a homogeneous space under a locally compact group  $G$.  Then, up to a factor of proportionality, there exists at most one non-trivial relatively invariant positive integral with modulus $\Delta$.
\end{theorem}

\section{Characterization of complex and quaternionic ellipsoids}
Let $V$ be an $n$-dimensional inner product space  over $\scf$. Let $B(V)= \{v\in V\colon \|v\|=1\}$ denote the  euclidean unit ball in $V$. Depending on the scalar field $\scf$, we call $E$ a real, complex, or quaternionic ellipsoid in $V$ if $E$ is the image of the euclidean unit ball in $V$ under an  $\scf$-affine  map $f\colon V\to V$.

We will make use of the following useful characterization of complex ellipsoids, which was established in \cite{A-B-M}.
\begin{theorem}[Arocha--Bracho--Montejano]\label{thm-complex-ellipsoids}
Let $V$ be a finite-dimensional complex vector space and $K$ be a convex body in $V$. Then, $K$ is a complex ellipsoid if and only if for every complex affine line $l$ of $V$ the intersection $l\cap K$ is either empty, a point, or a euclidean disk.    
\end{theorem}

Our goal is to extend the above characterization to the quaternionic case.  We call a hermitian matrix $A\in \Mat(n,\HH)$ positive definite if $x^*Ax>0$ for all non-zero $x\in \HH^n$.  

\begin{lemma} \label{lemma:ellipsoid} For $E\subseteq \scf^n$ the following are equivalent:
	\begin{enumerate}[(a)]
		\item $E$ is an ellipsoid with non-empty interior.
		\item There exist $a\in \scf^n$ and a positive definite Hermitian matrix $H$ such that
		$$ E=\{x\in \scf^n\colon \langle x-a,H(x-a)\rangle \leq 1\}.$$
	\end{enumerate}
If these properties are satisfied, then $a$ and $H$ are unique.
\end{lemma}
\begin{proof}
	That (a) implies (b) is straightforward.  The reverse implication follows from the existence of a square root for positive semidefinite hermitian matrices, which remains valid over the quaternions, see Corollary~6.2 and Remark~6.1 in \cite{Zhang:Quaternions}. Uniqueness in the complex and quaternionic case  follows from the real case. 
\end{proof}

\begin{theorem}\label{thm-quaternionic-ellipsoids} Let $V$ be a finite-dimensional quaternionic vector space and let  $K$ be a convex body in $V$. Then, $K$ is a quaternionic ellipsoid if and only if for every quaternionic affine line $L$ of $V$, $L\cap K$ is either empty, a point,  or a $4$-dimensional euclidean ball.  
\end{theorem}
\begin{proof}
	
If $K$ is a quaternionic ellipsoid, then it is straightforward to see that the intersection of $K$ with any quaternionic line is either empty, a point,  or a $4$-dimensional euclidean ball.

Let us prove the reverse statement. Let us point out  that with $K$ also every translate of $K$ does  satisfy the hypothesis. Choose a unit norm,  purely imaginary quaternion  $\xi\in \HH$. Then $\xi^2=-1$ and $J(v)= v\cdot \xi$ defines a complex structure on $V$. If  $v\in V$ is non-zero, then $\{v(a+\xi b)\colon a,b\in \RR\}$ is a complex line in $(V,J)$ and every complex line in $(V,J)$ is of this form. Therefore each (affine) complex line $l$ in $(V,J)$ is contained an (affine) quaternionic line $L$ in $V$.  Since $L\cap K$ is either empty, a point,  or a $4$-dimensional euclidean ball, the intersection $l\cap K$ must be either empty, a point, or a euclidean disc. Thus, by Theorem \ref{thm-complex-ellipsoids}, $K$ is a complex ellipsoid. In particular, replacing $K$ by a translate if necessary, there exists a self-adjoint, positive definite $\RR$-linear  endomorphism $H$ of $V $ such that 
$$K=\{ v\in V \colon \langle v, Hv\rangle \leq 1\}.$$ 

 Repeating the above argument for each purely imaginary quaternion $\xi\in S^3\subseteq \HH$ shows in particular that  $K\cdot w=K$ holds for every $w\in S^3$.  Thus, by the uniqueness part of Lemma~\ref{lemma:ellipsoid}, 
 $$ H(v\cdot w)\cdot w^{-1} = H(v)$$ 
 holds for all $v\in V$ and $w\in S^3$. We conclude that $H$ is  quaternionic-linear in addition to being real self-adjoint and positive definite. Together with Lemma~\ref{lemma:ellipsoid} this finishes the proof. 

\end{proof}

\section{The Busemann random simplex inequality}\label{sec-BRS}
For this section, $\scf$ will denote $\mathbb{C}$ or $\mathbb{H}$ and $p\in\{2,4\}$ will denote the dimension of $\scf$ over $\RR$. Let $\Phi:\RR_+\to \RR$ be a fixed strictly increasing function.

The first lemma  concerns  convex bodies  in euclidean space  $\RR^n$ without additional structure. In the special case $m=2$ it follows from a  result of Pfiefer \cite{Pfiefer:Max}.
\begin{lemma}\label{lemma-C_1}
	Let $K_1,\dots,K_m$ be convex bodies  in $\RR^n$. Set 
	$$M(K_1,\dots,K_m)=\int_{x_1\in K_1}\dots\int_{x_n\in K_m}\Phi(\|x_1+\cdots+x_m\|)dx_1\cdots dx_m.$$
Then 
	\begin{equation}\label{eq-C_1-1}
		M(K_1,\dots,K_m) \geq M(B_1,\dots,B_m),    
	\end{equation}
	where $B_i$ is the  euclidean ball with center at the origin and volume equal to the  volume of $K_i$.

Moreover, if the convex bodies $K_1,\ldots,K_m$ have non-empty interior, then equality holds in \eqref{eq-C_1-1} if and only if the  $K_i$ are euclidean balls and the centers $z_i$ of these balls satisfy 
$$ z_1+\cdots+z_m=0.$$

 \end{lemma}
\begin{proof}
	The proof is a modification of the argument in \cite{Groemer}.  Without loss of generality we can assume $\Phi(0)=0$. Moreover, it will be convenient for us to assume that $\Phi$ is continuous at $0$ and continuous from the left in $(0,\infty)$. This does not add any restrictions to the problem, since changing the value of $\Phi$ at countable many points will not affect the integrals involved in \eqref{eq-C_1-1}.

 We first show that for every linear hyperplane $H\subseteq \RR^n$, 
	\begin{equation}\label{eq:SteinerM}
		M(K_1,\dots,K_m) \geq M(\St_H K_1,\dots,\St_H K_m). 
	\end{equation}
	 Let $P_H$ denote the orthogonal projection onto $H$ and let $u$ be a unit vector perpendicular to $H$. For each point $y\in P_H(K_i)$ there are numbers  $a_i=a_i(y)$  and $w_i=w_i(y)\geq 0$ such that 
	$$ K_i \cap (y+ H^\perp)= \{y+tu \colon  t\in a_i  + [-w_i,w_i]\}.$$
We may clearly assume that $u=e_n=(0,\dots,0,1)$. Set
$P:=P_H(K_1)\times\dots\times P_H(K_m)$. For $Y=(y_1,\dots,y_m)\in P$ and $s\geq 0$, define
\begin{eqnarray*}&&v(Y)=(a_1(y_1),\dots,a_m(y_m))\in \mathbb{R}^m,\\ &&E(Y)=[-w_1(y_1),w_1(y_1)]\times\dots\times[-w_m(y_m),w_m(y_m)]\subseteq \mathbb{R}^m,\\
&&F(Y,s)=\{z=(z_1,\dots,z_m)\in \mathbb{R}^m:\Phi(\|(y_1,z_1)+\cdots+(y_m,z_m)\|)\leq s\}.
\end{eqnarray*}	
 Let $0\leq s <\sup \Phi(\RR_+)$.
Notice that due to the assumption on $\Phi$, $F(Y,s)$ is the closed symmetric strip $$ \{z\in\mathbb{R}^m:|\langle z,(1,1,\dots,1)\rangle|\leq (\Phi^{-1}(s^*)^2-\|y_1+\dots+y_m\|^2)^{1/2}\}, $$
where $s^*:=\sup\{s'\leq s:s'\in\Phi(\RR_+)\}$. Observe also that, due to the fact that $\Phi$ is continuous from the left, $s^*\in\Phi(\RR_+)$. Moreover, $s^*=s$, if $s\in\Phi(\RR_+)$. Finally, if $s\geq \sup \Phi(\RR_+)$, then $F(Y,s)=\RR^m$.
	
Set $K_i^0:=S_H K_i$ and $K_i^1:=K_i$, $i=1,\dots,m$ and let $\varepsilon\in\{0,1\}$. By Fubini's theorem and the layer cake representation we have 
\begin{eqnarray*}
M(K_1^\varepsilon,\dots,K_m^\varepsilon)&=&\int_P\int_{E(Y)}\Phi(\|(y_1,z_1) +\cdots+(y_m,z_m)\|dzdY\\
&=&\int_P\int_0^\infty |(E(Y)+\varepsilon v(Y))\setminus F(Y,s)|dsdY\\
&=&\int_P\int_0^{\infty}(|E(Y)|-|(E(Y)+\varepsilon v(Y))\cap F(Y,s)|dsdY.
\end{eqnarray*} 
Thus, in order to establish \eqref{eq-C_1-1}, it suffices to prove that for $Y\in P$ and $s\geq 0$, it holds
\begin{equation}\label{eq-St-s-local}
|(E(Y)+ v(Y))\cap F(Y,s)|\leq |E(Y)\cap F(Y,s)|.
\end{equation}
Due to the convexity of $E(Y)$ and $F(Y,s)$, it holds
$$E(Y)\cap F(Y,s)\supseteq\frac{1}{2}[(E(Y)+v(Y))\cap F(Y,s)]+\frac{1}{2}[(E(Y)-v(Y))\cap F(Y,s)],$$while the symmetry of $E(Y)$ and $F(Y,s)$ yields $|(E(Y)+v(Y))\cap F(Y,s)|=|(E(Y)-v(Y))\cap F(Y,s)|$. Equation 
\eqref{eq-St-s-local} follows immediately by the previous observations and the Brunn-Minkowski inequality.

This finishes the proof of  \eqref{eq:SteinerM} and successive  Steiner symmetrizations imply inequality \eqref{eq-C_1-1}.

Next, assume that equality holds in \eqref{eq-C_1-1}. Then, equality must hold in \eqref{eq:SteinerM} and therefore equality holds in \eqref{eq-St-s-local}, for almost every $(Y,s)\in P\times[0,\infty)$. We argue that equality holds in \eqref{eq-St-s-local} for every choice of $Y\in P$ and $s\in \Phi(\RR_+)$. To this end, 
first notice that if for some $s\geq 0$ equality holds in \eqref{eq-St-s-local} for almost every $Y\in P$, then a continuity argument shows that equality must hold for {\it every} $Y\in P$. 
Set ${\cal D}\subseteq \RR_+$ to be the set of all $s\geq 0$, such that equality holds in \eqref{eq-St-s-local} for the pair $(Y,s)$, for all $Y\in P$. Then, as $\RR_+\setminus {\cal D}$ has measure zero, for a fixed  $s\in \Phi(\RR_+)$, there is a sequence $\{s_j\}$ from ${\cal D}$, such that $s_j\searrow s$. It is clear that for all $j$ it holds $s\leq s_j^*\leq s_j$, hence $s_j^*\to s$. Recall that the function $\Phi^{-1}:\Phi(\RR_+)\to\RR$ is continuous. Since $s^*_j\in \Phi(\RR_+)$, we conclude that $\Phi^{-1}(s^*_j)\to\Phi^{-1}(s)$. Thus for $\varepsilon\in\{0,1\}$ and for all $Y\in P$, we have
\begin{eqnarray*}
|(E(Y)+\varepsilon v(Y))\cap F(Y,s_j)|&=&  |(E(Y)+\varepsilon v(Y))\cap F(Y,s^*_j)|\\
&\to&|(E(Y)+\varepsilon v(Y))\cap F(Y,s)| 
\end{eqnarray*}and our claim follows.

For an arbitrary $Y=(y_1,\ldots,y_m)\in\textnormal{int}\, P$, set $$s:=\Phi(\|(y_1,w_1(y_1))+\dots+(y_m,w_m(y_m))\|)\in \Phi(\RR_+).$$ Then, $F(Y,s)$ contains $E(Y)$ and at the same time its boundary touches the boundary of $E(Y)$. Due to equality in \eqref{eq-St-s-local}, $F(Y,s)$ also contains $E(Y)+v(Y)$.
Evidently, this can only happen if 
$\langle v(Y), (1,1,\dots,1)\rangle=0$. 

Thus, we have shown that equality in \eqref{eq:SteinerM} implies
 
\begin{equation}\label{eq:centers} a_1(y_1)+\cdots +a_m(y_m)=0,\end{equation} for each 
interior point $y_i \in P_H K_i$. This shows  that each $a_i$ is a constant function of $y_i\in P_HK_i$.  Geometrically this means that $K_i$ is symmetric with respect to a hyperplane parallel to $H=e_n^\perp$. Since this is true for all linear hyperplanes $H$, it follows from Lemma~\ref{thm:ellipsoids} that $K_i$ is a euclidean ball. Let $z_i$ denote the center of the ball $K_i$. By considering  coordinate hyperplanes $H$  in equation  \eqref{eq:centers}, one obtains  $z_1+\cdots + z_m=0$. 
\end{proof}

An immediate consequence of Lemma \ref{lemma-C_1} is the following.
\begin{lemma}\label{lemma-C}
	Let $K_1,\dots,K_m$ be convex bodies in $\scf$ and $\lambda=(\lambda_1,\dots,\lambda_m)\in\scf^m$. Let $M_\lambda(K_1,\dots,K_m)$ denote the integral $$\int_{x_1\in K_1}\dots\int_{x_m\in K_m}\Phi(|x_1\cdot \lambda_1+\dots +x_m\cdot \lambda_m|)dx_1\dots dx_m.$$
	Then
	\begin{equation}\label{eq-C}
		M_\lambda(K_1,\dots,K_m)\geq M_\lambda(B_1,\dots,B_m),
	\end{equation}
	where $B_i$ is the euclidean ball with center at the origin and volume equal to the volume  of $K_i$.
	
	Let $J=\{ j \colon \lambda_j\neq 0\}$. If the convex bodies $K_1,\ldots, K_m$ have non-empty interior and equality holds in \eqref{eq-C}, then $K_j$ is for each $ j\in J$ a euclidean ball  and the centers $z_j$ of these balls satisfy 
	$$ \sum_{j\in J} z_j\cdot \lambda_j=0.$$
	
\end{lemma}

\begin{proof}

Notice that \eqref{eq-C} holds trivially as equality if $\lambda_1=\dots=\lambda_m= 0$, hence we may assume that not all $\lambda_i$ are equal to $0$. 
Suppose that  $\lambda_1,\dots,\lambda_l$ are all non-zero and while $\lambda_{l+1}=\dots=\lambda_m=0$. Then Lemma~\ref{lemma-C_1} yields 
\begin{align}\label{eq-lemma-C-last}\begin{split}
	M_\lambda(K_1,\dots,K_m) 
	&= |\lambda_1\cdots \lambda_l|^{-p} |K_{l+1}|\cdots|K_m| 	M(K_1\cdot \lambda_1,\dots,K_l\cdot \lambda_l) \\
&\geq |\lambda_1\cdots \lambda_l|^{-p} |K_{l+1}|\cdots|K_m| 	M(|\lambda_1| B_1,\dots,|\lambda_l|B_l)\\
&=M_{\lambda} (B_1,\ldots,B_m).\end{split}
\end{align}

If equality holds in \eqref{eq-C}, then equality must hold in \eqref{eq-lemma-C-last}. Lemma~\ref{lemma-C_1} applied to $K_1\cdot \lambda_1,\ldots, K_l\cdot \lambda_l$  implies the equality conditions.  
\end{proof}

The following key proposition establishes the monotonicity of $\mathscr{B}$ under   symmetrization with respect to 
complex or quaternionic linear hyperplanes.  
\begin{proposition}\label{prop:symm}Let $H\subseteq \scf^n$ be a complex or quaternionic linear hyperplane and  let $K$ be a convex body in $\scf^n$. Then  
\begin{equation}\label{eq-S-L}
\mathscr{B}(K_1,\ldots, K_n) \geq \mathscr{B}(\St_{H}K_1,\ldots, \St_{H}K_n).
\end{equation}
Moreover, if the bodies  $K_1,\ldots, K_n$ have non-empty interior and  equality holds in \eqref{eq-S-L}, then for any $(y_1,\ldots, y_n)$ belonging to a certain dense open subset of $P_H K_1\times \cdots \times P_H K_n$,  the intersections 
$ K_i\cap (y_i+H^\perp)-y_i\subseteq H^\perp$ are euclidean balls with centers $z_i$ satisfying 
$$ z_1\cdot \lambda_1 +\cdots + z_n \cdot \lambda_n=0$$
for certain non-zero scalars $\lambda_1,\ldots, \lambda_n\in \scf$ that depend as in Proposition~\ref{prop:row_exp} only on the $(n-1)\times n$-matrix with columns the coordinates of  $(y_1,\ldots, y_n)$  with respect to some orthonormal basis.
\end{proposition}

\begin{proof} Choose coordinates so that $H=\{z\in \scf^n\colon z_1=0\}$.  
For a point  $x\in\scf^n$, we write $x=(a,y)$, where $a\in H^\perp=\scf$ and $y\in H=\scf^{n-1}$. Put $$K_i(y)=\{ a\in \scf\colon (a,y)\in  K_i \cap (y+ H^\perp)\}.$$ 
 Notice that for any $y_1,\ldots, y_n\in H$ there exist  by Proposition~\ref{prop:row_exp} scalars $\lambda_1,\dots,\lambda_n\in \scf$ depending only on $y_1,\ldots, y_n$ such that  
\begin{align}\label{eq:innerIntegral}\begin{split} 
\int_{a_1\in K_1(y_1)}\dots\int_{a_n\in K_n(y_n)}&\Phi( |\det ((a_n,y_n),\ldots,(a_n, y_n))|)\,da_1\cdots da_n\\
&= M_{\lambda}(K_1(y_1),\ldots,K_n(y_n))
\end{split} \end{align}
More precisely, the scalars $\lambda_i$  satisfy $|\lambda_i|= |\det A(1,i)|$, where $A$ denotes the matrix with columns $(a_i,y_i)$.  The identity \eqref{eq:innerIntegral} shows in particular that the function 
$$ (y_1,\dots,y_n)\mapsto  M_{\lambda}(K_1(y_1),\ldots,K_n(y_n))$$
 is in fact continuous in the interior of $P_H K_1\times \cdots\times P_H K_n $.

By Fubini's Theorem, one has
\begin{align*}
\mathscr{B}(K_1,\ldots, K_n) &=\int_{y_n\in P_H K_n }\cdots\int_{y_1\in P_H K_1 }  M_{\lambda}(K_1(y_1),\ldots,K_n(y_n)) dy_1\cdots d y_n 
\end{align*}
 Lemma~\ref{lemma-C} shows
\begin{equation}\label{eq-C-last}
M_{\lambda}(K_1(y_1),\ldots,K_n(y_n))  \geq M_{\lambda}(B_1(y_1),\ldots,B_n(y_n)),
\end{equation}
where $B_i(y_i)$ denotes the euclidean ball in $H^\perp$ with center at the origin and volume equal the volume of $K_i(y_i)$. From \eqref{eq-C-last} one immediately obtains inequality  \eqref{eq-S-L}.

Assume now that equality holds in \eqref{eq-S-L}. Then equality must hold in \eqref{eq-C-last}, for almost every $(y_1,\dots,y_n)\in P_H K_1\times \cdots\times P_H K_n $. 
By continuity, equality holds in \eqref{eq-C-last} for all $(y_1,\dots,y_n)$ in the interior of $P_H K_1\times \cdots\times P_H K_n$. The subset of those $A\in \Mat((n-1)\times n)$  such that at  least one  $(n-1)\times (n-1)$-submatrix is singular is by Proposition~\ref{prop:real_det} a real algebraic set. Thus there is an open dense subset of $P_H K_1\times \cdots\times P_H K_n$ such that all $(n-1)\times (n-1)$-submatrices of the $(n-1)\times n$ matrix formed by the columns $y_1,\dots, y_n$ are non-singular. By Proposition~\ref{prop:row_exp} this implies that the scalars $\lambda_1,\ldots, \lambda_n$ are all non-zero. Thus by Lemma \ref{lemma-C}, each  $K_i(y_i)\subseteq \scf$  has to be a euclidean ball and the centers of these balls have to satisfy $z_1\cdot \lambda_1 +\cdots + z_n\cdot \lambda_n=0$.  
\end{proof}

\begin{proof}[Proof of Theorem \ref{thm-BRS}] 

Choose a finite collection of (complex or quaternionic) linear hyperplanes of $\scf^n$ satisfying the hypothesis of  Theorem \ref{Bianchi-Gardner-Gronchi}. For example, the coordinate hyperplanes plus one additional hyperplane will do. Thus we obtain a sequence of symmetrizations  such  that $\St_{H_m}\circ\dots\circ S_{H_1}K_i$ converges for each $i\in \{1,\ldots,n\}$ to a euclidean ball centered at the origin. Since by Proposition~\ref{prop:symm} these symmetrizations  do not increase $\mathscr{B}$, it follows easily that \eqref{eq-thm-BRS} holds.

We are left with investigating when  equality occurs in \eqref{eq-thm-BRS}. Let $K_1,\ldots, K_n$  be convex bodies with non-empty interior  satisfying  $\mathscr{B}(K_1,\ldots, K_n)=\mathscr{B}(B_1,\ldots,B_n)$. Then due to the minimality of $\mathscr{B}(K_1,\ldots, K_n)$ and by \eqref {eq-S-L}, for any linear hyperplane $H$, equality holds in \eqref{eq-S-L}.  Thus, by Proposition~\ref{prop:symm} combined with an approximation argument, the intersection of $K_i$ with any affine (complex or quaternionic) line is either empty, a point or a euclidean ball.  This, together with Theorem~\ref{thm-complex-ellipsoids} and Theorem~\ref{thm-quaternionic-ellipsoids}, yields that each $K_i$ must be a complex or quaternionic ellipsoid. 

To finish  the proof of Theorem \ref{thm-BRS}, it remains to show that the ellipsoids  $K_1,\ldots, K_n$ are homothetic and centered at the origin. By choosing the coordinate system appropriately, we may assume that $K_1$ is the  translate of a euclidean ball, say $K_1= B+ z$. Let $H$ be a linear hyperplane perpendicular to $z$. Then $K_1(y)$ is for every  $y\in P_H K_1$ a euclidean ball with center $z$. By the minimality of the $K_i$, equality holds in \eqref{eq-S-L}. By Proposition~\ref{prop:symm}, the centers of the balls $K_i(y_i)$ satisfy 
$$ z_1\cdot \lambda_1  = - z_2\cdot \lambda_2 - \cdots - z_n \cdot \lambda_n$$
for $(y_1,\ldots, y_n)$ belonging to a certain open dense subset. 
For fixed $y_2,\ldots, y_n$, Lemma~\ref{lemma:row_exp} shows that $\lambda_1$ can be chosen independently of $y_1$, while  $\lambda_j$, $j\geq 2$, can be chosen so that there is a non-zero element $\mu^{-1}\in \scf$ such that the $\lambda_j\cdot \mu^{-1}$, $j\geq 2$, are linearly independent linear functions of $y_1\in \scf^{n-1}$.    First, since $z_1=z$ is constant,  this implies $z_1\cdot \lambda_1=0$ and hence $z_1=0$. Consequently, the linear independence implies $z_2=\cdots = z_n=0$. Since this holds for all $H$, we conclude using Lemma~\ref{thm:ellipsoids} that the $K_i$ are euclidean balls centered at the origin. 
In other words, the proof of Theorem \ref{thm-BRS} is complete.
\end{proof}

For later purposes let us compute the value of $\mathscr{B}$ on euclidean balls for specific weight functions $\Phi$. 

\begin{lemma} \label{lemma:Bballs}Let $\Phi(t)=t^r$, $r\geq 1$, and let $B^{np}\subseteq \scf^n$ be the euclidean unit ball. Then 
	\begin{align*} \mathscr{B}(B^{np}, \ldots, B^{np}) & =  \int_{B^{np}} \cdots \int_{B^{np}} |\det(x_1,\ldots, x_n)|^r dx_1\cdots dx_n\\
		& = (\kappa_{np+r})^n \prod_{j=0}^{n-1} \frac{\omega_{(n-j)p}}{\omega_{(n-j)p+r}}
	\end{align*}

	\end{lemma}
\begin{proof}
	Put 
	$$ I_n  = \int_{\scf^n} \cdots \int_{\scf^n}  |\det(x_1,\ldots, x_n)|^r e^{-\sum_{i=1}^n |x_i|^2 } dx_1\cdots dx_n.$$ 
	Observe that 
	$$ I_1 = \frac{\omega_{p}}{\omega_{p+r}} \pi^\frac{p+r}{2}.$$
	By Fubini's theorem
	\begin{align*} 
	I_n & = \int_{\scf^n} |x_n |^{r  } e^{-|x_n|^2}  \left( \int_{(\scf^n)^{n-1}} |{\det}_{x_n^\perp}(P_{x_n^\perp} x_1, \cdots, P_{x_n^\perp} x_{n-1})  e^{-\sum_{i=1}^{n-1} |x_i|^2 }dx_1 \cdots dx_{n-1} \right) dx_n\\
	& = 
	\frac{\omega_{np}}{\omega_{np+r}} \pi^{\frac{np+r}{2}+ \frac{p(n-1)}{2}}	
	 \cdot  I_{n-1}  \\
	 &= \frac{\omega_{np}}{\omega_{np+r}} \pi^{\frac{r-p}{2}+ np}	
	 \cdot  I_{n-1}
		\end{align*} 
	and hence
	$$ I_n =  \pi^{\frac{n(np+r)}{2}} \prod_{j=0}^{n-1}\frac{\omega_{(n-j)p}}{\omega_{(n-j)p+r}}$$
	On the other hand,
	\begin{align*} I_n  &=  2^{-n}  \Gamma(\frac{np+r}{2})^n  \int_{S^{np-1}} \cdots \int_{S^{np-1}} |\det(u_1,\ldots, u_n)|^r du_1\cdots du_n\\
		&=  \Gamma(\frac{np+r}{2}+1)^n  \int_{B^{np}} \cdots \int_{B^{np}} |\det(x_1,\ldots, x_n)|^r dx_1\cdots dx_n\\
		& = \frac{\pi^\frac{n(np+r)}{2}}{(\kappa_{np+r})^n} \int_{B^{np}} \cdots \int_{B^{np}} |\det(x_1,\ldots, x_n)|^r dx_1\cdots dx_n.
				\end{align*}
	
\end{proof}

\section{The linear Blaschke--Petkantchin formula}\label{sec-BP}

Let $V$ be an $n$-dimensional inner product  space over $\scf$. Let   $m\leq n$ be an integer.  Suppose  that $(v_1,\ldots, v_m)$ is  a tuple of vectors from $V$. Choose an $m$-dimensional linear subspace $E$ containing the vectors $v_1,\ldots, v_m$. We define 
$$ |\det(v_1,\ldots, v_m)|= |\det f|,$$
where $f\colon E\to E$ is the linear map defined by $f(e_i)=v_i$ where $e_1,\ldots,e_m$ is an orthonormal basis of $E$. Note that if the vectors  $v_1,\ldots,v_m$ are linearly independent, then $E$ is unique; if they are linearly dependent, then any choice of $E$ will give $\det f=0$. Moreover, since unitary transformations satisfy $|\det f|=1$, the definition does not depend on the choice of orthonormal basis of $E$.

The general linear group $\GL(m,\scf)$ acts on the set $m$-tuples of vectors from $E$ from the  right by 
$$ (v_1,\ldots, v_m)A = (
\sum_{r=1}^m v_r A_{r1}, \ldots,\sum_{r=1}^m v_r A_{rm}).$$

\begin{lemma}\label{lemma:det_right_action} For $A\in \GL(m,\scf)$ one has
	$$ |\det((v_1,\ldots, v_m)A)|= |\det A| |\det(v_1,\ldots, v_m)|.$$
\end{lemma}
\begin{proof} Fix an orthonormal basis $e_1,\ldots, e_m$ of a subspace $E$ containing the vectors $v_1,\ldots, v_m$. If $Q\in \Mat(m,\scf)$ denotes the matrix  with coefficients $Q_{ij} = \langle e_i,v_j\rangle$, then
	$ |\det(v_1,\ldots, v_m)| = |\det Q|$. Thus 
	$$|\det((v_1,\ldots, v_m)A)| = |\det (QA)| = |\det A| |\det(v_1,\ldots, v_m)|,$$
	as claimed. 
\end{proof}

\begin{theorem}\label{thm-BP}
	Let $m\in\{1,\dots,n-1\}$ and let $p\in \{1,2,4\}$  denote the dimension of  $\scf$ over $\R$. 	For any Borel measurable  function $f:(\scf ^n)^m\to [0,\infty)$
	\begin{align*} &\int_{\scf^n}\dots\int_{\scf^n} f(x_1,\dots,x_m)dx_1\dots dx_m\\
		& =c_{m,n}\int_{\Gr_{m}(n,\scf )}\int_E\dots\int_Ef(x_1,\dots,x_m)|\det(x_1,\dots,x_m)|^{(n-m)p}dx_1\dots dx_mdE,\end{align*}
	where 
	$$c_{m,n}=\prod_{j=0}^{m-1} \frac{\omega_{(n-j)p}}{\omega_{(m-j)p}}  $$
\end{theorem}

\begin{proof}
	The proof follows closely the proof of the Blaschke--Petkantchin formula in the real setting, see \cite[Abschnitt 7.3]{SchneiderWeil:StochastischeGeom}. Set $G=\U(n,\scf )\times \GL(m,\scf )$ and $X$ to be the set of all linearly independent $m$-tuples from $\scf ^n$. Define the action
	$$G\times X\ni ((U,A),(x_1,\dots,x_m))\mapsto U(x_1,\dots,x_m)A^{-1}.$$
	It is clear that $G$ acts transitively on $X$. 
	
	Observe that the complement of $X$ is contained in the vanishing set of all $m\times m$ minors of the matrix with columns $(x_1,\ldots,x_m)$. In particular, the complement of $X$ has zero Lebesgue measure.
	Let $I_1,I_2:C_c(X)\to\R$ be the functionals defined by
	$$I_1(f)=\int_{\scf ^n}\dots\int_{\scf ^n}f(x_1,\dots,x_m)dx_1\dots dx_m$$
	and
	$$I_2(f)=\int_{\Gr_m(n,\scf)}\int_E\dots\int_Ef(x_1,\dots,x_m)|\det(x_1,\dots,x_m)|^{(n-m)p}dx_1\dots dx_mdE.$$ 
	
	Fix $(U,A)\in G$ and $f\in C_c(X)$. Using Proposition~\ref{prop:real_det} and  Lemma~\ref{lemma-determinant}, we compute
	\begin{eqnarray*}
		I_1((U,A)\cdot f)&=&\int_{\scf ^n}\dots\int_{\scf ^n}f(U^{-1}(x_1,\dots,x_m)A)dx_1\dots dx_m\\
		&=&\int_{\scf ^n}\dots\int_{\scf ^n}f((x_1,\dots,x_m)A)dx_1\dots dx_m\\
		&=&|\det A|^{-np} \; I_1(f).
	\end{eqnarray*}
	On the other hand, notice that
	$$ |\det(U^{-1} x_1,\ldots, U^{-1} x_m)|  = |\det( x_1,\ldots, x_m)|$$
	for $U\in \U(n,\scf)$. Using Lemma~\ref{lemma:det_right_action} and Lemma~\ref{lemma-determinant}, we obtain 
	\begin{align*} &I_2((U,A)\cdot f)\\&  =\int_{\Gr_m(n,\scf)}\int_E\dots\int_Ef((U^{-1}x_1,\dots,U^{-1} x_m)A)|\det(U^{-1}x_1,\dots,U^{-1}x_m)|^{(n-m)p}dx_1\dots dx_mdE\\
		& = \int_{\Gr_m(n,\scf)}\int_E\dots\int_E f((x_1,\dots, x_m)A)|\det(x_1,\dots,x_m)|^{(n-m)p}dx_1\dots dx_mdE\\
		& =|\det A |^{-(n-m)p}\int_{\Gr_m(n,\scf)}\int_E\dots\int_E f((x_1,\dots, x_m)A)|\det((x_1,\dots,x_m)A)|^{(n-m)p}dx_1\dots dx_mdE\\
		& =|\det A|^{-np}\int_{\Gr_m(n,\scf)}\int_E\dots\int_E f(x_1,\dots, x_m)|\det(x_1,\dots,x_m)|^{(n-m)p}dx_1\dots dx_mdE
	\end{align*}

	We conclude that
	 $I_1$ and $I_2$ are relatively invariant positive integrals with the same modulus. Theorem~\ref{thm-uniqueness} implies that $I_1$ and $I_2$ are proportional.
	
To determine the factor of proportionality, we choose $f$ to  be the indicator function of the product of the euclidean unit balls  $B^{np}\times \cdots \times B^{np}$ and use Lemma~\ref{lemma:Bballs}.
\end{proof}

We are now ready to prove Theorem~\ref{thm-BI}. 

\begin{proof}[Proof of Theorem~\ref{thm-BI}] Applying the Blaschke--Petkantchin formula to the function $f(x_1,\ldots,x_m)=  \mathbf 1_{K_1\times \cdots \times K_m}(x_1,\ldots, x_m)$ yields
	$$ |K_1|\cdots |K_m| = c_{m,n} \int_{\Gr_m(n,\scf)} 
	 \mathscr{B}(K_1\cap E,\cdots, K_m\cap E) 
	 dE$$
for  $\Phi(x)= x^{(n-m)p}$. Since $\Phi$ is homogeneous, the
 Busemann random simplex inequality (Theorem~\ref{thm-BRS}) and Lemma~\ref{lemma:Bballs}  imply
\begin{equation}\label{eq:app_BRI}\mathscr{B}(K_1\cap E,\cdots, K_m\cap E)\geq b_{m,n}
	  |K_1\cap E|^{n/m}\cdots |K_m\cap E|^{n/m}\end{equation} 
with the  constant 
$$b_{m,n}= \frac{(\kappa_{np})^m}{(\kappa_{mp})^n} \prod_{j=0}^{m-1} \frac{\omega_{(m-j)p}}{\omega_{(n-j)p}}
.$$  This yields   the desired inequality.  

If equality holds in \eqref{eq:thm-BI}, then equality holds in \eqref{eq:app_BRI} for all $E\in \Gr_m(n,\scf)$. If $m=1$, then the intersection of $K_1$ with each line $E$ is a centered euclidean ball. In other words, $K_1$ is invariant under multiplication by  scalars of absolute value $1$. If $m\geq 2$, then the bodies $K_1\cap E, \ldots, K_m\cap E$ are homothetic ellipsoid centered at the origin. Since every real $2$-dimensional linear subspace is contained in some subspace $E$, this implies first   that the bodies  $K_i$  are ellipsoids centered at the origin. As in real case one concludes that these ellipsoids must be homothetic.

\end{proof}

\section{Affine quermassintegrals}
\label{sec:invariance}

In the real case, the results of this section were established by Grinberg~\cite{Grinberg:Isoperimetric} using the theory of invariant integration on locally compact groups. Moreover, it is mentioned on page 79 of \cite{Grinberg:Isoperimetric} that a routine modification of this argument yields the complex case.   
Here we reprove the invariance of the affine quermassintegrals using the description of the density line bundle of the Grassmannian stated in Proposition~\ref{prop:densities}. Our proof works simultaneously for reals, the complex numbers, and the quaternions.

Throughout this section, let $p\in \{1,2,4\}$ denote the dimension of $\scf$ over $\RR$.
For every convex body $K$  in $\scf^n$ with non-empty interior  and $m\in \{1,\ldots, n\}$ we define
the  (real, complex, quaternionic) \emph{$m$-th affine  quermassintegral} of $K$ as 
$$ \Phi_m^\scf(K)= \int_{\Gr_m(n,\scf)} |P_E K|^{-n} dE.$$
Here $P_E\colon \scf^n\to E$ denotes the orthogonal projection and $|P_E K|$ denotes the (euclidean) volume of $P_EK$ in $E$.

Replacing projections by intersections, we define the  (real, complex, quaternionic) \emph{$m$-th dual affine quermassintegral} of $K$ as 
$$\wt \Phi_m^\scf(K)= \int_{\Gr_m(n,\scf)} |K\cap E|^{n} dE.$$
where $|K\cap E|$ denotes the (euclidean) volume of $K\cap E$ in $E$. 

\begin{theorem}\label{thm-invariance} Let $m\in\{1,\ldots,n\}$. 
	For every convex body $K$  in $\scf^n$ with non-empty interior  and every $g\in \SL(n,\scf)$ it holds
	$$ \Phi_m^\scf(gK)=  \Phi_m^\scf(K)\quad \text{and}\quad  \wt \Phi_m^\scf(gK)= \wt \Phi_m^\scf(K).$$
\end{theorem}
We remark that $\wt \Phi_m^\scf(K)$ is defined and  invariant under $\SL(n,\scf)$ maps even if $K$ is assumed to be only compact.
\begin{proof}[Proof of Theorem \ref{thm-invariance}]
	Let us first treat the case of the dual affine quermassintegral, since it is slightly easier. 
	Let $\wt \tau_m$ denote the line bundle over $\Gr_m(n,\scf)$ with fiber 
	$$\wt \tau_m|_E=  \Dens(T_E \Gr_m(n,\scf))\otimes \Dens(E)^{\otimes n}$$ 
	over $E$. By Proposition~\ref{prop:densities}, there is a $\Stab_{\GL(n,\scf)}(E)$-equivariant 
	isomorphism
	$$ \wt \tau_m|_E\simeq  \Dens(V)^{\otimes m}.$$
	Thus the action of $\operatorname{Stab}_{\SL(n,\scf)}(E)$ on the fiber $\wt \tau_m|_{E}$ is trivial and  hence there exists a non-trivial $\SL(n,\scf)$-invariant section $\wt \sigma_m$ of  $\wt \tau_m$.  Since $\wt \sigma_m$ is in particular $\U(n,\scf)$-invariant, we may assume that   
	\begin{equation}\label{eq:inv_sec_dual} \wt \sigma_m =\rho_m \otimes (\vol_E)^{\otimes n}\end{equation}
	where $\rho_m$ is the $\U(n,\scf)$-invariant probability density on the Grassmannian and $\vol_E$ is the euclidean volume on $E$. 
	
	Observe that  \eqref{eq:inv_sec_dual} can be paired with any section $ \wt \kappa$ of the bundle over $\Gr_m(n,\scf)$ with fiber $\Dens(E^*)^{\otimes n}$ and integrated over $\Gr_m(n,\scf)$. One such section is $\wt \kappa(K)_E = (K\cap E)^{\otimes n}$, which yields  the identity
	$$ \int_{\Gr_{m}(n,\scf )} \langle \wt \sigma_m, \wt \kappa(K) \rangle = \wt \Phi_m^\scf (K).$$
	Since all bundles are equivariant, $\wt \kappa(gK) = g\cdot \wt\kappa(K)$, $g\cdot \wt\sigma_m=\wt\sigma_m$, and the pairing is compatible with the group action, the invariance of $\wt \Phi_m$ follows.

	The case of the affine quermassintegral is parallel to the former. 
	Let $\tau_m$ denote the line bundle over $\Gr_{n-m}(n,\scf)$ with fiber 
	$$\tau_m|_F=  \Dens(T_F \Gr_{n-m}(n,\scf))\otimes \Dens((V/F)^*)^{\otimes n}$$ 
	over $F$.  As before we obtain that the section 
	\begin{equation}\label{eq:inv_sec}  \sigma_m =  \rho_{n-m} \otimes (\vol_{V/F}^*)^{\otimes n}\end{equation}
	where $\rho_{n-m}$ is the $\U(n,\scf)$-invariant probability density on the Grassmannian and $\vol_{V/F}$ is the euclidean volume on $V/F\simeq F^\perp$, is $\SL(n,\scf)$-invariant. 
	
	Observe that  \eqref{eq:inv_sec} can be paired with any section $\kappa$ of the bundle over $\Gr_{n-m}(n,\scf)$ with fiber $\Dens(V/F)^{\otimes n}$ and integrated over $\Gr_{n-m}(n,\scf)$. One such section is $\kappa(K)_F = (\frac{1}{\vol_{V/F}(P_{V/F} K)} \vol_{V/F} )^{\otimes n}$.
	Thus 
	$$  \int_{\Gr_{n-m}(n,\scf)} \langle \sigma_m, \kappa(K) \rangle = \int_{\Gr_{n-m}(n,\scf)} |P_{F^\perp}K|^{-n} dF= \Phi_m^\scf(K).$$
	As before, the  invariance of $\Phi_m$ follows.	
\end{proof}

We close this section with a proof of Conjecture~\ref{conj} in one specific case. 

\begin{proposition}\label{prop:special}
	Let $K$  be a convex body in $\scf^n$ with non-empty interior. If $K$ is invariant under multiplication by scalars of unit length, then 
	$$ |K|^{-1} \geq \frac{(\kappa_{p})^n}{\kappa_{np}} \int_{\Gr_{1}(n,\scf)} |P_E K|^{-n} dE$$
	with equality if and only if $K$ is a (real, complex, or quaternionic) ellipsoid.
\end{proposition}
\begin{proof} 
	If $E$ is a line  through the origin in $\scf^n$ and $K$  is invariant under multiplication by scalars of unit length, then 
	$P_E K$ is a euclidean ball  of radius $h_K(u)$, where $u$ is any vector in $E$ of unit length. Hence
	\begin{align*}  (\kappa_{p})^n \int_{\Gr_{1}(n,\scf)} |P_E K|^{-n} dE & = \frac{1}{\omega_{np}}  \int_{S^{np-1}}  h_K(u)^{-np} du \\
		& = \frac{1}{np \kappa_{np}}  \int_{S^{np-1}}  \rho_{K^*}(u)^{np} du\\
		&  = \frac{1}{ \kappa_{np}} |K^*|,
	\end{align*} 
where $\rho_K$ denotes the radial function of $K$ and $K^*$ is the polar body of $K$. An application of the Blaschke--Santal\'o inequality yields the desired inequality.
\end{proof}

\bibliographystyle{abbrv}
\bibliography{papers}
\vspace{1.5 cm}
\noindent\\
\hspace*{1em}Christos Saroglou, Department of Mathematics,
University of Ioannina,
Ioannina, Greece, 45110\\
\hspace*{1em}E-mail address: csaroglou@uoi.gr \ \& \ christos.saroglou@gmail.com
\\
\\
\noindent 
\hspace*{1em}Thomas Wannerer, Friedrich-Schiller-Universit\"at Jena, Fakult\"at f\"ur Mathematik und Informatik, Institut f\"ur Mathematik, Ernst-Abbe-Platz 2, 07743 Jena, Germany\\
\hspace*{1em}E-mail address: thomas.wannerer@uni-jena.de 

\end{document}